\documentclass[12pt, reqno]{amsart}
\usepackage{amssymb}
\usepackage{amsrefs}
\usepackage{amsmath}
\usepackage{bbm}
 \newtheorem{thm}{}[section]
 \newtheorem{theorem}[thm]{Theorem}
 \newtheorem{corollary}[thm]{Corollary}
 \newtheorem{lemma}[thm]{Lemma}
 \newtheorem{proposition}[thm]{Proposition}
 \theoremstyle{definition}
 \newtheorem{definition}[thm]{Definition}
 \theoremstyle{remark}
 \newtheorem{remark}[thm]{Remark}

 \numberwithin{equation}{section}
 \allowdisplaybreaks
 
\newcommand{\NN}{\ensuremath{\mathbb{N}}}

\newcommand{\XX}{\ensuremath{\mathbb{X}}}

\newcommand{\xx}{\ensuremath{\mathbf{e}}}
\newcommand{\BB}{\ensuremath{\mathcal{B}}}
\newcommand{\GG}{\ensuremath{\mathcal{G}}}
\newcommand{\gr}{\ensuremath{\mathbf G}}

 \newcommand{\sgn}{\mathop{\mathrm{sign}}\nolimits}
 \newcommand{\supp}{\mathop{\mathrm{supp}}\nolimits}
\newcommand{\spn}{\mathop{\mathrm{span}}\nolimits}
\newcommand{\co}{\mathop{\mathrm{co}}\nolimits}

\begin{document}

\title{Characterization of $1$-almost greedy bases}
\author[F. Albiac]{F. Albiac}\address{Mathematics Department\\ Universidad P\'ublica de Navarra\\Campus de Arrosad\'{i}a\\
 Pamplona\\ 31006 Spain\\
 Tel.: +34-948-169553\\
 Fax: +34-948-166057\\
} \email{fernando.albiac@unavarra.es}
 \author[J. L. Ansorena]{J. L. Ansorena}\address{Department of Mathematics and Computer Sciences\\
Universidad de La Rioja\\ Edificio Luis Vives\\
 Logro\~no\\
26004 Spain\\
 Tel.: +34-941299464 \\
 Fax: +34-941299460} \email{joseluis.ansorena@unirioja.es}



\begin{abstract} This article closes the cycle of  characterizations of greedy-like bases in the ``isometric'' case initiated in \cite{AlbiacWojt2006} with the characterization of $1$-greedy bases and continued in \cite{AAJAT} with the characterization of $1$-quasi-greedy bases.  Here we settle  the problem of providing a characterization of $1$-almost greedy bases in Banach spaces.  We show that a basis in a Banach space is almost greedy with almost  greedy constant equal to $1$ if and only if it has Property (A). This fact permits now to state that a basis is $1$-greedy if and only if it is $1$-almost greedy and $1$-quasi-greedy. As a by-product of our work we also provide a tight characterization of almost greedy bases.
  \end{abstract}

%\thanks{}

\maketitle

\section{Introduction and background}\label{intro}
\noindent Suppose $(\XX, \Vert\cdot\Vert)$ is an infinite-dimensional  separable  (real or complex) Banach space and  $\BB=(\xx_n)_{n=1}^\infty$ is a 
(Schauder)  basis for $\XX$. We will denote by $(\xx_n^*)_{n=1}^\infty$ the corresponding  biorthogonal functionals. Assume that  $\BB$ is semi-normalized, i.e.,
$0<\inf_n \Vert \xx_n\Vert \le \sup_n \Vert \xx_n\Vert<\infty$. Each $x\in \XX$ has a unique  series expansion in terms of the basis,
\[
x=\sum_{n=1}^\infty \xx_n^*(x)\xx_n.
\]
The {\it support} of an  $x\in \XX$ is the set supp$(x)=\{n\in \NN\colon \xx_{n}^{\ast}(x)\not=0\}$. 

A \textit{greedy ordering} for $x$ is an injective mapping
$\rho\colon\NN\to \NN$  such that  $\supp(x)\subset \rho(\NN)$  and $|\xx_{\rho(i)}^{\ast}(x)|\ge |\xx_{\rho(j)}^{\ast}(x)|$ if $i\le j$. For each $m\in \NN$, an
\textit{$m$-term greedy approximation} to $x$, denoted by $\gr_m(x)$, is the $m$-th partial sum of the (formal) rearranged series $ \sum_{i=1}^{\infty}\xx_{\rho(i)}^{\ast}(x)\xx_{\rho(i)}$ for some greedy ordering $\rho$ of $x$, i.e.,
\[
\gr_m(x)=  \sum_{i=1}^{m}\xx_{\rho(i)}^{\ast}(x)\xx_{\rho(i)}.
\]  
If   the sequence $(\xx_n^*(x))_{n=1}^\infty$ contains several terms with the same absolute value then the greedy ordering for $x$ is not unequivocally determined.  However,
there is a unique greedy ordering $\rho$ for $x$ fulfilling the extra properties that $\rho(i)\le\rho(j)$ if $|\xx_{\rho(i)}^{\ast}(x)|= |\xx_{\rho(j)}^{\ast}(x)|$ 
and $\rho(\NN)=\NN$ if $x$ is finitely supported (notice that  $\rho(\NN)=\supp (x)$ is $x$ is infinitely supported). We will refer to this ordering  as the {\it natural greedy ordering} for $x$.
With this convention, the {\it $m$-th natural greedy 
approximation}  to $x$ is given by
$$
   \GG_{m}[\BB, \XX](x):=\GG_{m}(x) = \sum_{i=1}^{m}\xx_{\rho(i)}^{\ast}(x)\xx_{\rho(i)},
$$
where $\rho$ is the natural greedy ordering.

Konyagin and Temlyakov  \cite{KonyaginTemlyakov1999} defined a basis to be \textit{ greedy} if $\GG_{m}(x)$ is essentially the \textit{best $m$-term approximation} to $x$ using basis vectors, i.e., there exists a constant $C\ge 1$ such that, for all $x\in \XX$ and $m\in \NN$, 
\begin{equation}\label{defgreedy}
\Vert x-\GG_{m}(x)\Vert \le C\inf\{\Vert x-\sum_{n\in A}\alpha_{n}\xx_{n}\Vert\colon |A|=m, \,  \alpha_{n}\text{ scalars}\}.
\end{equation}
The smallest $C$ in \eqref{defgreedy} is the \textit{greedy constant} of the basis and will be denoted by $C_{g}$.
They then showed that greedy bases can be intrinsically characterized as unconditional bases with the additional property of being \textit{democratic}, i.e., $\Vert \sum_{n\in A} \xx_{n}\Vert\le \Delta \Vert \sum_{n\in B}  \xx_{n}\Vert$ whenever $|A|=|B|$, for some constant $\Delta\ge 1$. Recall also  that a basis  $(\xx_n)_{n=1}^\infty$  is \textit{unconditional}   if for $x\in \XX$ the series $\sum_{i=1}^\infty \xx_{\pi(i)}^*(x)\xx_{\pi(i)}$ converges  to $x$ for any permutation $\pi$ of $\NN$. 

The property of being unconditional is easily seen to be equivalent to that of being \textit{suppression unconditional}, which means that there is a constant $K$ such that for all $x\in\XX$
and all $A\subset \NN$, $A$ finite, 
\begin{equation}\label{unifbound}
\left\Vert x- P_A(x)\right\Vert \le K \left\Vert x \right\Vert,
\end{equation}
where
 $
 P_A(x)=\sum_{n\in A} \xx_n^*(x) \xx_n
$
is the natural projection onto the linear span of
$\{\xx_{n}\colon n\in A\}$.
The smallest $K$ in \eqref{unifbound} coincides with the least constant  $K$ such that for all $x\in \XX$ and $ A\subset \NN$ finite we have
\[
\Vert P_A(x)\Vert   \le K\Vert x\Vert.
\] 
 This optimal constant is called the \textit{suppression unconditional constant} of the basis and is denoted by $K_{su}$.  
If $\BB$ is unconditional and $K\ge 1$ is  such that $K_{su}\le K$ we say that $\BB$ is \textit{$K$-suppression unconditional}.

 A basis which is democratic and unconditional is \textit{superdemocratic}, i.e., there exists a best constant $\Gamma\ge 1$ (\textit{$\Gamma$-superdemocratic}) such that the inequality
\[
\left\Vert \sum_{n\in A} \varepsilon_{n} \xx_{n}\right\Vert\le \Gamma \left\Vert \sum_{n\in B}\theta_{n} \xx_{n}\right\Vert
\]
holds for any two finite sets of integers $A$ and $B$ of the same cardinality, and any choice of signs $(\varepsilon_{n})_{n\in A}$ and $(\theta_{n})_{n\in B}$.

Another key concept in the theory, introduced as well by
Konyagin and Temlyakov  \cite{KonyaginTemlyakov1999},
 is that of \textit{quasi-greedy} basis. A basis is quasi-greedy if  for $x\in \XX$, $\lim_{m\to \infty}\GG_{m}(x)=x$, that is,  the  series $\sum_{i=1}^\infty \xx_{\rho(i)}^*(x)\xx_{\rho(i)}$ converges  to $x$, where $\rho$ is the natural greedy ordering for $x$.   Subsequently, Wojtaszczyk \cite{Wo2000} proved that these are precisely the bases for which the greedy operators $(\GG_m)_{m=1}^{\infty}$ are uniformly bounded (despite the fact that they are nonlinear on $x$), i.e., there exists a constant $C\ge 1$ such that for all $x\in \XX$ and $m\in \NN$, 
\begin{equation}\label{WoCondition}
\Vert \GG_m(x)\Vert\le C \Vert x\Vert.
\end{equation}
We will denote by $C_w$  the smallest constant in \eqref{WoCondition}. 

Obviously, if \eqref{WoCondition} holds  then there is a (possibly different) least constant $C_{\ell}\ge 1$ such that 
\begin{equation}\label{SupWoCondition}
\Vert x-\GG_m(x)\Vert\le {C_{\ell}} \Vert x\Vert, \quad x \in \XX, \,  m\in\NN.
\end{equation}
 Some authors call $C_{w}$  the \textit{quasi-greedy constant} of the basis, whereas others give that name to the number   \[
 C_{qg}=\max\{C_{w}, C_{\ell}\}.
 \]
For the time being, and while  a satisfactory consensus is reached,  by analogy with unconditional bases we will call the number $C_{\ell}$ in \eqref{SupWoCondition}  the \textit{suppression quasi-greedy constant} of the basis. Thus from now on  we will use \textit{$C$-suppression quasi-greedy} to refer to a quasi-greedy basis  such that \eqref{SupWoCondition} holds with a constant $C\ge C_{\ell}$. Regardless of one's preferences for the \textit{right} quasi-greedy constant, 
  if $\BB$ is   $K$-suppression unconditional  then $\BB$ is quasi-greedy with $C_{qg}\le K$.   

As  Wojtaszczyk pointed out in \cite{Wo2000}, the choice of the greedy ordering for each $x\in \XX$ with which to construct the greedy operators $(\gr_{m})_{m=1}^{\infty}$ plays no relevant role in the theory. 
Indeed, if a basis $(\xx_n)_{n=1}^\infty$ is quasi-greedy   then $x=\sum_{i=1}^\infty \xx^*_{\rho(i)}(x) \xx_{\rho(i)}$ for all $x\in\XX$
and all possible greedy orderings $\rho$ of $x$. Similarly,  if $(\xx_n)_{n=1}^\infty$
 is $C$-suppression quasi-greedy  then for all $x\in\XX$ and $m\in \NN$ we have
 \begin{equation}\label{eq:Csupcond}
 \Vert x-\gr_m(x)\Vert \le C \Vert x\Vert,\end{equation}
for any $m$-term greedy approximation $\gr_m(x)$ to $x$.

 Dilworth et al.\ introduced in \cite{DKKT2003} a  property for bases basis that is intermediate between quasi-greedy and greedy.
They defined a basis $\BB=(\xx_n)_{n=1}^\infty$  to be \textit{ almost greedy} if there exists a constant $C\ge 1$ such that for all $x\in \XX$ and $m\in \NN$,
 \begin{equation}\label{defalmostgreedy}
\Vert x-\GG_{m}(x)\Vert \le C\inf \{\Vert x- P_A(x) \Vert\colon |A|=m \}.
\end{equation}
Comparison with \eqref{defgreedy} shows that this is formally a weaker condition: in \eqref{defgreedy} the infimum is taken over all possible linear combinations that we can form with $m$ basis elements whereas in \eqref{defalmostgreedy} only projections of $x$ onto $m$-term subsets of $\BB$ are considered. 
 The least constant in \eqref{defalmostgreedy} is the \textit{ almost greedy constant} of the basis and is denoted by $C_{ag}$. If $\BB$ is an almost greedy basis and $C$ is a constant such that $C_{ag}\le C$ we say that $\BB$ is \textit{ $C$-almost greedy}.

 In this paper we are concerned with greedy-like bases in the ``isometric" case, i.e.,   in the case that the  constants that arise in  the context of greedy  bases in the three above-mentioned different forms  are $1$. This study was initiated in \cite{AlbiacWojt2006}, where the authors obtained the following characterization of $1$-greedy bases.
 
 \begin{theorem}[\cite{AlbiacWojt2006}*{Theorem 3.4}]\label{1GCharacterization} A basis $\BB$ for a Banach space $\XX$ is $1$-greedy if and only if $\BB$ is  $1$-suppression unconditional  and satisfies Property (A).
 \end{theorem}
  
 In order to explain this characterization we need a few more definitions. 
 
Given a basis $\BB=(\xx_n)_{n=1}^\infty$  for $\XX$ and  $x$, $y$ in $\XX$ we say that $y$ is a \textit{greedy permutation} of $x$  if we can write
\begin{equation}\label{eq:greedypermutation}
x=z+t \sum_{n\in A} \varepsilon_{n} \xx_{n}\quad\text{and}\;\; y=z+ t \sum_{n\in B}\theta_{n} \xx_{n}
\end{equation}
for   some $z\in\XX$,
some sets of integers $A$ and $B$ of the same finite cardinality with $\supp(z)\cap (A\cup B)=\emptyset$,  some signs $(\varepsilon_{n})_{n\in A}$ and $(\theta_{n})_{n\in B}$, and some scalar $t$ such that
 $\sup_n |\xx_n^*(z)|\le t$. If, in addition, $A \cap B=\emptyset$, we say that $y$ is a 
 \textit{disjoint greedy permutation} of $x$.
 In other words, $y$ is obtained from $x$ by moving  those  terms of $x$ (or some of them) whose coefficients are  maximum in absolute value to  gaps in the support of $x$. We are also allowed to change the sign of (some of) the terms we move. Then, the basis
$\BB$ is said to  satisfy \textit{Property (A)} if $\Vert x\Vert=\Vert y\Vert$ whenever $y$ is a  disjoint greedy permutation of $x$.

 Property (A) can be relaxed by allowing a factor of distortion in the norm of vectors which are a disjoint greedy permutation of each other, which motivates the next concept. 
 
 \begin{definition}  A basis $\BB$ of a Banach space $\XX$ is said to be \textit{symmetric for largest coefficients} if there is a constant $C\ge 1$ (\textit{$C$-symmetric for largest coefficients}) such that  if $x$ and  $y$ are finitely supported vectors in $\BB$ we have $
 \Vert y\Vert \le C\Vert x\Vert$ whenever $y$ is a disjoint greedy permutation of $x$.
\end{definition} 

Note that for $C=1$ this is just  the above mentioned Property (A).  It is not hard to prove that
if $\BB$ is $C$-symmetric for largest coefficients and $y$ is a greedy permutation of $x$ then
 $\Vert y\Vert\le C^2 \Vert x\Vert$.
In particular, if $\BB$ has Property (A) and $y$ is a greedy permutation of $x$ then $\Vert y \Vert =\Vert x\Vert$ (which is the way Property (A) was originally defined in \cite{AlbiacWojt2006}).

The concept of symmetry for largest coefficients  is referred to as \textit{Property (A) with constant $C$} in \cite{DKOSZ2014}, 
 where the authors show the following generalization of Theorem~\ref{1GCharacterization}.
\begin{theorem}[\cite{DKOSZ2014}*{Theorem 2}]\label{DKOSGeneralization}
Let  $\BB$ be a basis of a Banach space $\XX$.  \begin{enumerate}
\item[(i)]  If  $\BB$ is $C$-greedy,
then it is  $C$-suppression unconditional and $C$-symmetric for largest coefficients.
 \item[(ii)] Conversely, if $\BB$  is $K$-suppression unconditional and and $C$-symmetric for largest coefficients, then it is $K^2C$-greedy.
\end{enumerate}
\end{theorem}

In turn, $1$-quasi-greedy bases have been recently characterized in \cite{AAJAT}.
\begin{theorem}[\cite{AAJAT}*{Theorem 2.1}]\label{1QGCharacterization}  Let $\BB$ be a basis in a Banach space $\XX$. The following are equivalent:
\begin{enumerate}
\item[(i)] $\BB$ is quasi-greedy with $C_{qg}=1$.
\item[(ii)] $\BB$ is quasi-greedy with $C_{w}=1$.
\item[(iii)] $\BB$ is suppression unconditional with suppression unconditional constant $K_{su}=1$.

\end{enumerate}

\end{theorem} 

Our aim is to complete the description of ``isometric" greedy-like bases by providing the following characterization  of almost greedy bases in the optimal case that $C_{ag}=1$.

 \begin{theorem}[Main Theorem]\label{1AGTheorem} A basis  in a Banach space  is $1$-almost greedy   if and only if it has Property (A).
 \end{theorem}

\section{Proof of the Main Theorem}
\noindent
The proof of  Theorem~\ref{1AGTheorem} will rely on the following result. Notice its analogy with Theorem~\ref{DKOSGeneralization}.
 \begin{proposition}\label{CAGTheorem} Let $\BB$ be a basis in a Banach space $\XX$.
 \begin{enumerate}
 \item[(i)]If $\BB$ is $C$-almost greedy then $\BB$ is $C$-suppression quasi-greedy and    $C$-symmetric for largest coefficients.
 \item[(ii)] Conversely, if $\BB$ is $K$-suppression quasi-greedy and  $C$-symmetric for largest coefficients then $\BB$ 
 is $C K$-almost greedy.
  \end{enumerate}
  \end{proposition}

 In order to show Proposition~\ref{CAGTheorem} we will need  the full force of the hypotheses.  Part (i) uses an equivalent re-formulation of  the condition defining almost greedy bases which will give us for free a small chunk of what we aim to prove. 

\begin{lemma}\label{improvingAQ} Suppose $\BB=(\xx_{n})_{n=1}^{\infty}$ is  $C$-almost greedy. Then for  $x\in\XX$, $m\in\NN$,  and any $m$-term greedy approximation  $\gr_m(x)$ of $x$
 we have
\begin{equation*}
\Vert x - \gr_m(x)  \Vert  \le C \inf\Big\{ \Vert x -P_A(x)\Vert \colon 0\le |A|\le m\Big\}.
\end{equation*}
\end{lemma}
\begin{proof} Our aim  is to prove that
\begin{equation}\label{EqDefAlmGre}
\Vert x- \gr_m(x) \Vert  \le  C \Vert x -P_A(x)\Vert,
\end{equation}
for $x\in \XX$, $m\in \NN$, 
$\gr_m(x)$ $m$-term greedy approximation  to $x$,
 and $A\subset \NN$ with $|A|\le m$.  

To that end, we start by assuming two extra conditions.
Suppose first that 
 $|\xx^*_{\rho_{(m+1)}}(x)| < |\xx^*_{\rho_{(m)}}(x)|$, where $\rho$ is the natural greedy ordering for $x$.   In this case $\gr_m(x)=\GG_m(x)$, and we say that
$\GG_m(x)$  is a \textit{strictly greedy approximation} to $x$.  Suppose also that  $x$ is finitely supported.
Then  there is  $A\subset B\subset\NN$ with $|B|=m$  such 
$P_A(x)=P_B(x)$. Therefore
\[
\Vert x- \gr_m(x) \Vert  =
\Vert x-\GG_m(x)\Vert\le C  \Vert x -P_B(x)\Vert = C \Vert x -P_A(x)\Vert,
\]
and we are done. 

To obtain  \eqref{EqDefAlmGre}  in the general case we use a standard perturbation argument. 
Let $E=\{\rho(i) \colon 1\le i \le m\}$.
Given $\delta>0$ there is $y\in\XX$ so that $\Vert x-y\Vert\le\delta$,  $\GG_m(y)$ is a strictly greedy approximation to  $y$, and
$\gr_m(x)=P_E(x)=P_E(y)=\GG_m(y)$. Then 
there is some  finitely supported $z$ in $\XX$ such that $\Vert y-z\Vert \le \delta$, $\GG_m(z)$ is a strictly greedy approximation to $z$, and  $\GG_m(z)= 
  P_E(z)$. Hence,
  \begin{align*}
\Vert x- \gr_m(x) \Vert  & \le  \Vert x-z \Vert + \Vert z-\GG_m(z)\Vert + \Vert  \GG_m(z)-\gr_m(x)\Vert \\
& \le  C \Vert z -P_A(z )\Vert +  (1+\Vert P_E\Vert) \Vert x- z \Vert\\
& \le  C (\Vert z - x \Vert  +  \Vert  x-P_A(x) \Vert  +  \Vert  P_A(x) - P_A(z )\Vert) \\  &\quad+ (1+\Vert P_E\Vert) \Vert x- z \Vert\\ 
& \le  C   \Vert  x-P_A(x) \Vert +  ( 1+ C + \Vert P_E\Vert + C  \Vert P_A\Vert ) \Vert x- z \Vert\\ 
& \le  C   \Vert  x-P_A(x) \Vert +  2( 1+ C + \Vert P_E\Vert + C  \Vert P_A\Vert )  \delta.
 \end{align*}
Letting  $\delta$ tend to zero we obtain \eqref{EqDefAlmGre} and  the proof is over.
\end{proof}

The proof of Part (ii) in Proposition~\ref{CAGTheorem}  relies on some stability properties  of quasi-greedy bases with respect to the multiplication by {\it certain} bounded sequences.
 
\begin{lemma}\label{finitemultiplier} Suppose $(\xx_{n})_{n=1}^{\infty}$ is $C$-suppression quasi-greedy. Then for any  $N\in\NN$, any $N$-tuple 
$(a_i)_{i}^N$ such that  $(|a_i|)_{i}^N$ is non-increasing, and any injective map $\rho\colon\{1,\dots,N\}\to\NN$ we have
\[
\left\Vert \sum_{i=1}^{N}\lambda_{i} a_i \xx_{\rho(i)} \right\Vert\le C\left\Vert \sum_{i=1}^{N}  a_i \xx_{\rho(i)} \right\Vert,
\]
 for all multipliers $(\lambda_{i})_{i=1}^{N}$ such that $0\le \lambda_{1}\le \lambda_2\le\dots \le \lambda_{N}\le 1$.\end{lemma}

\begin{proof} 
Applying  \eqref{eq:Csupcond} to $\sum_{i=1}^N a_i \xx_{\rho(i)}$ we get the claim for all  $(\lambda_i)_{i=1}^N$  in the set
\[\mathcal S=\left\{(\underbrace{0,\dots,0}_{m}, \underbrace{1,1,\dots, 1}_{N-m})\colon 0\le  m\le N\right\}.\]
 Therefore the statement holds for all $N$-tuples in the convex hull  $\co(\mathcal S)$ of $\mathcal S$. But $\co(\mathcal S)$ is precisely the collection of all $(\lambda_{i})_{i=1}^{N}$ of the desired form. 
\end{proof}

From this lemma we infer the following (non-linear) multiplier boundedness theorem.

\begin{theorem}\label{multiplier}Suppose  $(\xx_{n})_{n=1}^{\infty}$ is  $C$-suppression quasi-greedy. Then for any  $x\in \XX$ and any  
greedy ordering $\rho$  of $x$,
\[
\left\Vert \sum_{i=1}^\infty\lambda_i \xx_{\rho(i)}^*(x)\xx_{\rho(i)}\right\Vert\le C \Vert x\Vert,
\]
whenever  $(\lambda_i)_{i=1}^{\infty}$ is a non-decreasing sequence of scalars with $0\le \lambda_i\le 1$ for all $i$.
\end{theorem}

\begin{proof} Lemma~\ref{finitemultiplier} yields that for all  $M$, $N \in\NN$ with $M\le N$ we have
\[
\left\Vert \sum_{i=M}^N \lambda_i \xx_{\rho(i)}^*(x)\xx_{\rho(i)}\right\Vert\le 
C \left\Vert \sum_{i=M}^N  \xx_{\rho(i)}^*(x)\xx_{\rho(i)}\right\Vert.
\]
 Hence $ \sum_{i=1}^\infty\lambda_i \xx_{\rho(i)}^*(x)\xx_{\rho(i)}$ is a Cauchy series and therefore  converges. Moreover 
\[
\left\Vert \sum_{i=1}^\infty\lambda_i \xx_{\rho(i)}^*(x)\xx_{\rho(i)}\right\Vert\le  C \limsup_N
\left\Vert \sum_{i=1}^N  \xx_{\rho(i)}^*(x)\xx_{\rho(i)}\right\Vert
= C\Vert x\Vert.
\]
\end{proof}

We will also need the following  re-formulation of the symmetry for largest coefficients. 

\begin{proposition}\label{ImprovingSLC} Suppose  $(\xx_{n})_{n=1}^{\infty}$ is  $C$-symmetric for largest coefficients. Then, for any  $x\in \XX$,
\[
\Vert x\Vert \le C\left\Vert x-P_A(x) + t\sum_{n\in B} \varepsilon_n \xx_n\right\Vert,
\]
whenever $A$ and $B$ are such that 
$0\le|A|\le |B|<\infty$ and $B\cap\supp x=\emptyset$,   $|\varepsilon_n|=1$ for all $n\in B$, and  $|\xx_n^*(x)| \le t$ for all $n\in\NN$.
\end{proposition}

\begin{proof}Assume, without lost of generality, that $A\subset \supp x$. By density, it suffices to prove the result when
$x$ is finitely supported. In this case, there is $A\subset D$ such that $|D|=|B|$, $D\cap B=\emptyset$ and $D\cap \supp x=A$. Our hypothesis gives
\begin{equation}\label{eq:ImprovedSLC}
\left \Vert x -P_A(x)+  u\right\Vert \le C\left\Vert x-P_A(x) + t\sum_{n\in B} \varepsilon_n \xx_n\right\Vert
\end{equation}
for all $u$ in the set
 \[
\mathcal U=\left\{ t \sum_{n\in D} \theta_n \xx_n \colon |\theta_n|=1\right\}.
\]
Consequently, \eqref{eq:ImprovedSLC} holds for any
$u\in\co(\mathcal U)$. Since
\[
\co(\mathcal U)=\left\{\sum_{n\in D} s_n \xx_n \colon |s_n|\le t \right\},
\]
we have that $P_A(x)\in\co(\mathcal U)$, and we are done. \end{proof}

We are now in a position to complete  the proof of Proposition~\ref{CAGTheorem} and, immediately after, that of Theorem~\ref{1AGTheorem}.
 \smallskip

\begin{proof}[Proof of Proposition~\ref{CAGTheorem}] (i) Assume $\BB$ is $C$-almost greedy. 
To prove that $\BB$ is $C$-suppression quasi-greedy, just take $A=\emptyset$ in Lemma~\ref{improvingAQ}.
In order to get that
$\BB$ is  $C$-symmetric for largest coefficients, we pick $x$, $y\in\XX$
 such that $y$ is a disjoint greedy permutation of $x$. Let  
$t$, 
$A$, $B$, $(\varepsilon_{n})_{n\in A}$, $(\theta_{n})_{n\in B}$, and   $z$  be as in \eqref{eq:greedypermutation}  and consider
\[
u=z+t \sum_{n\in A} \varepsilon_{n} \xx_{n} +  t \sum_{n\in B}\theta_{n} \xx_{n}.
\]
Let $m=|A|=|B|$. Then   $\gr_m(u):=t \sum_{n\in A} \varepsilon_{n} \xx_{n}$ is a 
$m$-term greedy approximation to $u$,
and   
$P_B(u)= t \sum_{n\in B}\theta_{n} \xx_{n}$. Hence, by Lemma~\ref{improvingAQ},
\[
\Vert y \Vert =\Vert u- \gr_m(u) \Vert \le C \Vert u-P_B(u)\Vert=C\Vert  x \Vert.
\]

\smallskip (ii)
 Let $x\in\XX$, $m\in\NN$, and $A\subset \NN$ with $|A|=m$. For $n\in\NN$ put  $a_n=\xx_n^*(x)$.  Pick out $B\subset \NN$ of cardinality $m$ such that $P_B(x)=\GG_m(x)$.  Denote $t=\min \{|a_{n}|\colon n\in B\}$ and  $q=|A\setminus B|=|B\setminus A|$.
 By Proposition~\ref{ImprovingSLC},
 \[
 \begin{aligned}
 \Vert x-\GG_m(x)\Vert & \le C \left\Vert x-\GG_m(x) - P_{A\setminus B} (x) + t\sum_{n\in B\setminus A} \sgn(a_n) \xx_n\right\Vert\\
 &= C \left\Vert x-  P_{A\cup B} (x) +  \sum_{n\in B\setminus A} \frac{t}{|a_n|}   a_n  \xx_n\right\Vert.
 \end{aligned}
 \]
 Since the $|a_n|\le t \le |a_k|$ for $n\in\NN\setminus (A\cup B)$ and $k\in B\setminus A$, there is  
  a greedy ordering  $\rho$ for $x-P_A(x)$ such that $B\setminus A=\{ \rho(i) \colon 1\le i \le q\}$. 
  Define a sequence $(\lambda_i)_{i=1}^\infty$ by 
\[
\lambda_i=\begin{cases} t/ |a_{\rho(i)}| & \text{ if } 1\le i \le q, \\  1 & \text{ if } q<i.\end{cases}
\]
We infer that  that $(\lambda_i)_{i=1}^\infty$ is non-decreasing. By Theorem~\ref{multiplier},
\begin{equation*}
 \left\Vert x-  P_{A\cup B} (x) +  \sum_{n\in B\setminus A} \frac{t}{|a_n|}   a_n  \xx_n\right\Vert 
 =\left\Vert  \sum_{i=1}^\infty  \lambda_i a_{\rho(i)} \xx_{\rho(i)}  \right\Vert
  \le K \Vert x-P_A(x)  \Vert.
\end{equation*}
 Combining we obtain
$
\Vert x-\GG_m(x)\Vert  \le C K  \Vert x- P_A(x)\Vert,
$
as desired.\end{proof}

\begin{remark}The same technique  used in the proof of Part (ii) of Proposition~\ref{CAGTheorem} enables us to amalgamate some of the steps in the proof,
achieved in  \cite{DKOSZ2014},  of
Theorem~\ref{DKOSGeneralization}.  This slight improvement leads to show  that, if $\BB$  is $K$-suppression unconditional and $C$-symmetric for largest coefficients, then $\BB$ is $CK$-greedy.
\end{remark}

\begin{proof}[Proof of Theorem~\ref{1AGTheorem}] 
 Let  $\BB=(\xx_n)_{n=1}^\infty$ be a basis for a Banach space $\XX$.
Appealing to Proposition~\ref{CAGTheorem} we need only show that 
if $\BB$ has Property (A) then
 $\BB$ is 
$1$-suppression quasi-greedy.
Let  $x\in\XX$,  $k\notin \supp(x)$,  and 
$s$  be  a scalar such that $|\xx_n^*(x)|\le |s|$ for all $n\in\NN$. 
Applying Proposition~\ref{ImprovingSLC}, in the case in which $A=\emptyset$ and $B=\{k\}$, we obtain
$
\Vert x\Vert \le \Vert x +   s \xx_k \Vert.
$
From here, we conclude the proof using a straightforward induction argument.\end{proof}

As an on-the-spot corollary we can  re-state Theorem~\ref{1GCharacterization} involving the three kinds of greedy-like bases in the isometric case.

\begin{corollary} A basis $\BB$ of a Banach space $\XX$ is $1$-greedy if and only if it is $1$-almost greedy and $1$-quasi-greedy.

\end{corollary}

\begin{remark} Dilworth et al.\ gave a  first characterization of  almost greedy bases as those bases that are quasi-greedy and democratic with coarse estimates of the constants involved (see \cite{DKKT2003}*{Theorem 3.3}). Proposition~\ref{CAGTheorem} reflects the fact that if we wish to attain the optimality in the almost greedy constant of a basis we need to replace the property of being democratic with that of being symmetric for largest coefficients. At the same time, in order to be able to strictly remain within the class of  almost greedy bases  we must use the suppression quasi-greedy constant $C_{\ell}$ instead of $C_w$ or $C_{qg}$.
  Indeed, as soon as the quasi-greedy constant $C_{w}$ comes into play and attains the value $1$, the basis is $1$-suppression unconditional
 (see Theorem~\ref{1QGCharacterization}) and therefore we have trespassed on greedy territory! 
This observation, in combination with the likeness in the definition of $C_{\ell}$ and the definitions of greedy and  almost greedy constants,
 conveys that $C_{\ell}$ might be the fair candidate to be called the quasi-greedy constant of the basis.
 \end{remark}
 
\begin{remark} An idea that was  implicit in  \cite{AlbiacWojt2006} and  that is present as well   in the work of Dilworth et al.\ \cite{DKOSZ2014}, is that being $1$-superdemocratic cannot supplant Property (A) to determine the isometric properties of greedy-like bases, even when combined with other stronger features of the basis.  Indeed, the basis $\BB$ in the Banach space of Example 5.4 of \cite{AlbiacWojt2006} is  $1$-lattice unconditional and $1$-superdemocratic but fails to be $1$-greedy because $\BB$ does not have Property (A). The same reason, failure of Property (A),  is the one that, according to Theorem~\ref{1AGTheorem}, prevents  a basis  
from being $1$-almost greedy.

\end{remark}

\begin{remark} Although throughout this paper the word \textit{basis} meant \textit{Schauder basis} we would like to  stress that all we have said here, as well as the other characterizations of greedy-like bases in the isometric cases, remains valid in the more general framework of semi-normalized bounded biorthogonal systems, that is,   sequences $(\xx_{n},\xx_{n}^{\ast})_{n=1}^\infty\subset \XX\times \XX^{\ast}$ such that
\begin{enumerate}
\item $\XX= \overline{\spn \{ \xx_{n}\colon n\in\NN\}}$.
\item $\xx_{n}^{\ast}(\xx_{k})=1$ if $k=n$ and $\xx_{n}^{\ast}(\xx_{k})=0$ otherwise.
\item  $\sup \{ \max\{\Vert \xx_n\Vert,\Vert \xx_{n}^{\ast}\Vert\} \colon n\in\NN\}<\infty$.
\end{enumerate}
This kind of bases is quite common and there are important examples such as the trigonometric system in the space $L_{1}$ which belong to this class and yet they are not Schauder bases. Following Wojtaszczyk's approach in his study of greedy-like bases for bounded biorthogonal systems    \cite{Wo2000}, in this paper we have written down the proofs so that they work in this more general setting.
\end{remark}

\subsection*{Acknowledgement}  Both authors partially supported by the Spanish Research Grant \textit{ An\'alisis Vectorial, Multilineal, y Aplicaciones}, reference number MTM2014-53009-P. The first-named author  also acknowledges the support of Spanish Research Grant \textit{ Operators, lattices, and structure of Banach spaces}, 
with reference MTM2012-31286.

% ------------------------------------------------------------------------

\begin{bibsection}
\begin{biblist}

\bib{AlbiacWojt2006}{article}{
   author={Albiac, F.},
   author={Wojtaszczyk, P.},
   title={Characterization of 1-greedy bases},
   journal={J. Approx. Theory},
   volume={138},
   date={2006},
   number={1},
   pages={65--86},
   issn={0021-9045},
  % review={\MR{2197603 (2006i:41026)}},
   % doi={10.1016/j.jat.2005.09.017},
}

\bib{AAJAT}{article}{
   author={Albiac, F.},
   author={Ansorena, J.L.},
   title={Characterization of 1-quasi-greedy bases},
   journal={arXiv: 1504.04368v1 [math.FA] },
  % volume={138},
   date={2015},
   %number={1},
   %pages={65--86},
  % issn={0021-9045},
  % review={\MR{2197603 (2006i:41026)}},
    %doi={},
}

\bib{DKOSZ2014}{article}{
   author={Dilworth, S. J.},
   author={Kutzarova, D.},
   author={Odell, E.},
   author={Schlumprecht, Th.},
   author={Zs{\'a}k, A.},
   title={Renorming spaces with greedy bases},
   journal={J. Approx. Theory},
   volume={188},
   date={2014},
   pages={39--56},
  % issn={0021-9045},
  % review={\MR{3274228}},
  % doi={10.1016/j.jat.2014.09.001},
}

\bib{DKKT2003}{article}{
   author={Dilworth, S.J.},
   author={Kalton, N.J.},
   author={Kutzarova, D.},
   author={Temlyakov, V. N.},
   title={The thresholding greedy algorithm, greedy bases, and duality},
   journal={Constr. Approx.},
   volume={19},
   date={2003},
   number={4},
   pages={575--597},
  % issn={0176-4276},
   %review={\MR{1998906 (2004e:41045)}},
   %doi={10.1007/s00365-002-0525-y},
}

\bib{DOSZ2011}{article}{
   author={Dilworth, S.J.},
   author={Odell, E.},
   author={Schlumprecht, Th.},
   author={Zs{\'a}k, A.},
   title={Renormings and symmetry properties of 1-greedy bases},
   journal={J. Approx. Theory},
   volume={163},
   date={2011},
   number={9},
   pages={1049--1075},
 %  issn={0021-9045},
 %  review={\MR{2832742}},
  % doi={10.1016/j.jat.2011.02.013},
}

\bib{KonyaginTemlyakov1999}{article}{
    author={Konyagin, S. V.},
    author={Temlyakov, V. N.},
     title={A remark on greedy approximation in Banach spaces},
   journal={East J. Approx.},
    volume={5},
      date={1999},
    number={3},
     pages={365\ndash 379},
      %issn={1310-6236},
    %review={MR1716087 (2000j:46020)},
}
 \bib{Wo2000}{article}{
 author={Wojtaszczyk, P.},
 title={Greedy algorithm for general biorthogonal systems},
 journal={J. Approx. Theory},
 volume={107},
 date={2000},
 number={2},
 pages={293--314},
 % issn={0021-9045},
 %review={\MR{1806955 (2001k:46017)}},
 % doi={10.1006/jath.2000.3512},
}

 \end{biblist}
\end{bibsection}
\end{document}